\documentclass[11pt]{amsart}
\usepackage[T1]{fontenc}
\usepackage{lmodern}
\usepackage{amsmath,amssymb,mathtools}
\usepackage{enumitem}
\usepackage{microtype}
\usepackage[hidelinks]{hyperref}

\newtheorem{theorem}{Theorem}[section]
\newtheorem{lemma}[theorem]{Lemma}
\newtheorem{proposition}[theorem]{Proposition}

\theoremstyle{remark}
\newtheorem{remark}[theorem]{Remark}
\newtheorem{question}{Question}[section]
\newtheorem{claim}[theorem]{Claim}

\newcommand{\T}{\mathbb T}
\newcommand{\N}{\mathbb N}
\newcommand{\Z}{\mathbb Z}

\newcommand{\Mcal}{\mathcal M}
\newcommand{\Scal}{\mathcal S}
\newcommand{\Bcal}{\mathcal B}
\newcommand{\Ocal}{\mathcal O}
\newcommand{\Rcal}{\mathcal R}
\newcommand{\Fcal}{\mathcal F}
\makeatletter
\@namedef{subjclassname@2020}{\textup{2020} Mathematics Subject Classification}
\makeatother

\title[Generic zero-entropy optimization]{Generic Zero-Entropy Optimization for Finitely Generated Nonlacunary Actions on the Circle}
\author{Hang Zhao}
\address{School of Mathematics and Big Data, Neijiang Normal University, Neijiang, Sichuan 641100, China}
\email{hangzhaoscu@163.com}
\thanks{ORCID: 0009-0003-2445-3477.}
\thanks{This work was partially supported by the National Natural Science Foundation of China (No. 123B2006).}
\hypersetup{pdftitle={Generic Zero-Entropy Optimization for Finitely Generated Nonlacunary Actions on the Circle},pdfauthor={Hang Zhao}}

\begin{document}
\begin{abstract}
Let $\Bcal$ be a real Banach space continuously embedded in $C(\T)$ and containing a H\"older space $C^{k,\alpha}(\T)$ as a dense subspace. For each finitely generated nonlacunary semigroup $\Scal_F$ of integer multiplications on the circle, the potentials whose maximizing measures all have zero entropy for every element of $\Scal_F$ form an open dense set in $\Bcal$. In each H\"older space $C^{k,\alpha}(\T)$, there is an open dense subset that works simultaneously for every such semigroup. Moreover, there is a dense $G_{\delta}$ subset of $\Bcal$ such that, for every potential in this set and every such semigroup $\Scal_F$, the maximizing measure is unique and has zero entropy.
\end{abstract}
\maketitle
\section{Introduction}

In ergodic optimization, one asks which invariant probability measures maximize the integral of a given continuous potential. For background, see the surveys of Bochi and Jenkinson~\cite{BochiICM,Jenkinson2006,JenkinsonSurvey}. For a $C^1$ expanding circle map, Morris~\cite[Theorem~2]{Morris2008} proved that a generic H\"older potential has a unique maximizing
measure of zero entropy. Contreras~\cite[Theorem~A]{Contreras} showed that, for an open dense set of Lipschitz potentials,
the unique maximizing measure is supported on a periodic orbit. For a real analytic expanding circle map and a real analytic potential,
Gao and Shen~\cite[Main Theorem]{GaoShen} proved that every maximizing measure has zero entropy unless the potential is analytically
cohomologous to a constant.

We study this problem for finitely generated nonlacunary semigroups of integer multiplications on the circle. Joint invariance changes the class of admissible measures, so results for individual maps do not directly apply. Rudolph--Johnson entropy rigidity~\cite{Rudolph,Johnson}, together with ergodic decomposition, shows that all maximizing measures have zero entropy whenever Haar measure is not maximizing. For each fixed finitely generated nonlacunary semigroup, we prove that the potentials whose maximizing measures all have zero entropy form an open dense set in a broad class of Banach spaces. For each H\"older space, we obtain a single open dense set that works for all such semigroups.

Let $\T\coloneqq \mathbb R/\Z$, equipped with the quotient metric, and let $\lambda$ be normalized Haar measure. For $n\ge1$, define $T_n(x)\coloneqq nx \pmod 1$. Given a nonempty finite set $F\subset\N_{\ge1}$, put $$\Scal_F\coloneqq\left\{\prod_{n\in F}n^{a_n}:a_n\in\Z_{\ge0}\right\}.$$ A multiplicative semigroup $\Sigma\subset\N_{\ge1}$ is \textbf{nonlacunary} if it is not contained in
$\{d^j:j\in\Z_{\ge0}\}$ for any integer $d\ge2$.

Let $\Mcal(\T)$ denote the set of Borel probability measures on $\T$, and define the set of $\Scal_F$-invariant measures by
$$\Mcal_F\coloneqq \{\mu\in \Mcal(\T) :(T_n)_*\mu=\mu \text{ for every }n\in F\}.$$
This set is nonempty, weak$^*$ compact, and convex. For a real-valued $f\in C(\T)$, define $$\beta_F(f)\coloneqq\max_{\mu\in \Mcal_F }\int_{\T} f\;d\mu,\;\;\;\;\;\Mcal_F^{\max}(f)\coloneqq \{\mu\in \Mcal_F:\ \int f\;d\mu=\beta_F(f)\}.$$

For an integer $k\ge0$ and $\alpha\in(0,1]$, define the \textbf{H\"older space} $C^{k,\alpha}(\T)\coloneqq\{f\in C^{k}(\T):[f^{(k)}]_{\alpha}<\infty\}$, equipped with the norm
$$\|f\|_{C^{k,\alpha}}\coloneqq \sum_{i=0}^k \|f^{(i)}\|_{\infty}+[f^{(k)}]_\alpha,\;\;\;\text{where } [f]_{\alpha}\coloneqq \sup_{x\neq y}\cfrac{|f(x)-f(y)|}{d_{\T}(x,y)^{\alpha}}\;.$$
We write $C^{\alpha}(\T):=C^{0,\alpha}(\T)$ for $\alpha\in(0,1)$. We say that a jointly invariant probability measure $\mu\in\Mcal_F$ has \textbf{zero entropy for $\Scal_F$} if its metric entropy is zero with respect to every $T_n$ with $n\in\Scal_F$. 

Our main result is the following:
\begin{theorem}\label{thm:main}
    Let $\Bcal $ be a real Banach space continuously embedded in $C(\T)$. Suppose that  $C^{k,\alpha}(\T)$  is a dense subspace of $\Bcal $ for some integer $k\ge0$ and $\alpha\in (0,1]$. Then  the following two statements hold:
  \begin{enumerate}[label=\textup{\roman*.}, leftmargin=1.5em, labelsep=0.4em, topsep=0.3em, itemsep=0.2em, parsep=0pt]
     \item[\rm{i}.] For each finitely generated nonlacunary semigroup $\Scal_F,$ there is an open and dense set $\Ocal_F \subset \Bcal $ such that for any $f\in \Ocal_F$, each $\mu\in \Mcal_F^{\max}(f)$ has zero entropy for $\Scal_F.$
        \item[\rm{ii}.] There is a dense $G_{\delta}$ set  $\Rcal \subset \Bcal$ such that,  for every  $f\in \Rcal$ and  every finitely generated nonlacunary semigroup $\Scal_F$, the set $\Mcal^{\max}_F(f)$ consists of a single measure, which has zero entropy for $\Scal_F$.  
      \end{enumerate}
\end{theorem}

 Theorem~\ref{thm:main}  applies to $C(\T),$ $C^k(\T)$, $C^{k,\alpha}(\T)$  and periodic Sobolev spaces $W^{r,p}(\T)$, $r\in \N_{\ge 1},$ $1\le p\le \infty$ with continuous  representatives.  When $\Bcal=C^{k,\alpha}(\T),$ the open dense set in part (i) of Theorem~\ref{thm:main} can be chosen independently  of the semigroup:
 \begin{theorem}\label{thm:holder}
     For every integer $k\ge0$ and $\alpha\in(0,1]$, there is an open dense set $\Ocal_{k,\alpha}\subset C^{k,\alpha}(\T)$ such that, for  every finitely generated nonlacunary semigroup $\Scal_F$ and every $f\in\Ocal_{k,\alpha}$, all measures in $\Mcal_{F}^{\max}(f)$ have zero entropy for $\Scal_F$. 
 \end{theorem}

The zero-entropy statement in Theorem~\ref{thm:main}(i) is proved by excluding Haar measure as a maximizer. The analytic input is a Ces\`aro decay estimate for rational-grid errors of H\"older functions. This estimate allows us to select grid sizes in prescribed arithmetic progressions. We then construct arbitrarily small H\"older perturbations for which the corresponding rational-grid averages exceed the Haar average throughout a fixed neighborhood of potentials. For each finitely generated semigroup, one of these grids supports a jointly invariant probability measure. Rudolph--Johnson entropy rigidity then yields the zero-entropy conclusion.

To prove the uniqueness statement in Theorem~\ref{thm:main}(ii), we apply the result of Jenkinson, Li, Liao and Zhang~\cite[Theorem~1]{JenkinsonLiLiaoZhang} to each fixed semigroup and then take a countable intersection. For a single continuous map,  Jenkinson \cite[Theorem~3.2]{Jenkinson2006} proved generic uniqueness of the maximizing measure. In constrained ergodic optimization, Motonaga and Shinoda~\cite{MotonagaShinoda} proved generic zero entropy under certain rotation-vector constraints, and   Guo and McGoff~\cite[Example~3.4 and Theorem~4.4]{GuoMcGoff} proved generic uniqueness under linear constraints.

By contrast, in $C(\T)$ equipped with the uniform norm, Proposition~\ref{prop:uni} shows that the interior of the common zero-entropy set for all semigroups in Theorem~\ref{thm:main} is precisely $\{f\in C(\T):f(0)>\int_{\T}f\;d\lambda\}$. Since this half-space is not dense in the uniform norm, Theorem~\ref{thm:holder} does not extend to $C(\T).$ 

Our proof of Theorem~\ref{thm:holder} relies on Rudolph--Johnson entropy rigidity for nonlacunary semigroups. For single expanding maps, Morris's result already yields a common dense $G_{\delta}$ set. We do not know whether this can be strengthened to a common open dense set:
\begin{question}\label{q:single-maps}
Fix $0<\alpha\le 1$. Does there exist an open dense set $\mathcal O_{\alpha}\subset C^{0,\alpha}(\T)$ such that,
for every $f\in\mathcal O_{\alpha}$ and every $n\ge2$, each $T_{n}$-maximizing measure of $f$ has zero entropy for $T_{n}$?
\end{question}

\section{Excluding Haar Measure}
For each positive integer $m$ and $f\in C(\T)$, set
$$
\nu_m\coloneqq \frac1m\sum_{j=0}^{m-1}\delta_{j/m},
\;\quad
Q_m(f)\coloneqq \int_{\T} f\,d\nu_m-\int_{\T} f\,d\lambda .
$$
In particular, $Q_1(f)=f(0)-\int_{\T}f\,d\lambda$, a fact we will use repeatedly.

The following lemma gives the grid estimates we need. 
\begin{lemma}\label{lm:holder}
 For every $0<\alpha\le1$ and every $f\in C^{0,\alpha}(\T)$, we have $$\lim_{N\to\infty}\frac1N\sum_{m=1}^{N}m^\alpha|Q_m(f)|=0.
$$
\end{lemma}
\begin{proof}
For each integer $M\ge 1$ and each $v\in C(\T),$ set 
$$A_M(v):=\left(\cfrac{1}{M}\sum_{m=M}^{2M-1}|Q_m(v)|^2\right)^{1/2}.$$
For $u\in L^1(\T)$ and $n\in \Z$, write $\hat u(n):=\int_{\T}u(x)e^{-2\pi inx}\;d\lambda(x) $. For $g\in C^2(\T)$, we have 
$$Q_m(g)=\sum_{k\neq 0}\hat g(km)=-\cfrac{1}{4\pi^2m^2}\sum_{k\neq 0}\cfrac{\hat{g''}(km)}{k^2}\;.$$

For each fixed $k\neq0$, the frequencies $kM,\ldots,k(2M-1)$ are distinct. Thus Minkowski's inequality and Bessel's inequality give
$$A_M(g)\le \cfrac{1}{4\pi^2M^{5/2}}\sum_{k\neq 0}\cfrac{1}{k^2}\left(\sum_{m=M}^{2M-1}|\hat{g''}(km)|^2\right)^{1/2}\le \cfrac{1}{12}M^{-5/2}\|g''\|_{L^2}\;.$$

For $h\in (0,1]$, define
$$g_h(x)\coloneqq \cfrac{1}{h^2}\int_{0}^h\int_{0}^hf(x+s-t)\;dsdt.$$
Then $g_h\in C^2(\T)$ and
$$g''_h(x)=\left(f(x+h)-2f(x)+f(x-h)\right)/h^2.$$
Moreover, $\|g_h-f\|_{\infty}\le[f]_{\alpha}h^{\alpha}$ and $\|g''_h\|_{L^2}\le2[f]_{\alpha}h^{\alpha-2}$. Since $|A_M(u)|\le2\|u\|_{\infty}$ for every $u\in C(\T)$,
$$A_M(f)\le A_M(g_h)+2\|f-g_h\|_{\infty}\le [f]_{\alpha}\left(\cfrac{1}{6}M^{-5/2}h^{\alpha-2}+2h^{\alpha}\right).$$

Choosing $h=M^{-5/4}$ gives $A_M(f)\le (13/6)\cdot[f]_{\alpha}M^{-5\alpha/4}.$ Hence there exists a constant $C_{\alpha}>0$, depending only on $\alpha$, such that 
\begin{equation}
    \label{equ:CFM}
\cfrac{1}{M}\sum_{m=M}^{2M-1}m^{\alpha}|Q_m(f)|\le (2M)^{\alpha}A_M(f)\le C_{\alpha}[f]_{\alpha}M^{-\alpha/4}.\end{equation}

Finally, let $B:=\lfloor \log_2N\rfloor$. Using \eqref{equ:CFM}, we obtain
\begin{align*}
 \frac1N\sum_{m=1}^{N}m^\alpha|Q_m(f)|&\le \cfrac{1}{N}\sum_{j=0}^{B}\sum_{m=2^j}^{2^{j+1}-1}m^\alpha|Q_m(f)|\\
 &\le \cfrac{C_{\alpha}[f]_{\alpha}}{N}\sum_{j=0}^B2^{j(1-\alpha/4)}\\
 &\le 3C_{\alpha}[f]_{\alpha}N^{-\alpha/4}.
\end{align*}
Letting $N\to\infty$ completes the proof.
\end{proof}

We now use Lemma~\ref{lm:holder} to exclude Haar measure. Throughout this section, let $\Fcal$ be the family of all nonempty finite sets $F\subset\N_{\ge1}$. For $F\in\Fcal$, write $N_F\coloneqq \prod_{n\in F}\;n$ and 
$$G_F\coloneqq \{f\in C(\T):\beta_F(f)>\int_{\T}f\;d\lambda\}.$$
For each integer $m\ge 2$, define the probability measure $\sigma_m\coloneqq(m\nu_m-\delta_0)/(m-1)$ and write
$$
R_m(f)\coloneqq Q_m(f)-\frac{Q_1(f)}{m}
 =\frac{m-1}{m}\left(\int_\T f\;d\sigma_m-\int_\T f\; d\lambda\right).
$$

\begin{lemma}\label{lm:excluehaar} 
For every $k\in\Z_{\ge0}$ and $\alpha\in(0,1]$, there is an open dense set $\Ocal_X$ in the H\"older space $X=C^{k,\alpha}(\T)$ such that
$\Ocal_X\subset G_F$ for every nonempty finite $F\subset \N_{\ge 1}$.
\end{lemma}
\begin{proof} 
We will show that the open set
$$
\Ocal_X\coloneqq
\operatorname{int}_X
\left(
\bigcap_{\substack{F\in \Fcal}}
(X\cap G_F)
\right)
$$
is dense in $X$.
For every $0<\alpha\le1$ and every $\alpha$-H\"older function $u$, integrating over the grid intervals gives
$
m^\alpha|Q_m(u)|\le[u]_\alpha.
$
Since each arithmetic progression has positive density, Lemma~\ref{lm:holder} implies
\begin{equation}
\label{equ:mod}
\liminf_{\substack{m\to\infty\\m\equiv1\pmod q}}
m^\alpha|Q_m(u)|=0,\quad\text{for every integer } q\ge1.
\end{equation}
We split the proof into two cases. First, suppose that $k+\alpha\ge1$, so every function in $X$ is Lipschitz. The set
$$
\mathcal  U\coloneqq\{f\in X:Q_1(f)\ne0\}
$$
is open and dense because $Q_1$ is a nonzero continuous linear functional on $X$. If $Q_1(f)>0$, then $\delta_0\in\Mcal_F$ implies $f\in G_F$ for every $F\in \Fcal$.
If $Q_1(f)<0$, fix $F\in \Fcal$ and choose an integer satisfying 
$$m\ge2,\qquad m\equiv1\pmod{N_F},
\qquad m|Q_m(f)|<-Q_1(f).$$
Such an integer exists by \eqref{equ:mod} with $\alpha=1$. Since $\gcd(m,N_F)=1$, we have $\sigma_m\in\Mcal_F$. Moreover,
$$
\int_\T f\,d\sigma_m-\int_\T f\,d\lambda
=\frac{mQ_m(f)-Q_1(f)}{m-1}>0.
$$
Thus $\mathcal U\subset G_F$ for every $ F\in\Fcal$. Since $\mathcal U$ is open in $X$, it is contained in $\Ocal_X$, 
which is therefore dense.

It remains to consider the case $X=C^{\alpha}(\T)$ with $\alpha\in(0,1)$. Fix $f\in X$ and $\delta>0$. Set $m_0=1$ and, for each integer $j\ge1$, define recursively  
$$m_j\coloneqq \min\{m\in \N:m\ge 2m_{j-1},\;m\equiv 1\pmod j,\;m^{\alpha}|Q_m(f)|<\delta\}.$$
Such a sequence $\{m_j\}_{j\ge 1}$ exists by \eqref{equ:mod} with $q=j$. Define
$$H(x)\coloneqq \sum_{j\ge 1}m_j^{-\alpha}\cos(2\pi m_jx).$$
Since $m_j\ge2^j$, $H$ is continuous and 
$\|H\|_{\infty}\le \sum_{j\ge 1}2^{-j\alpha}=1/(2^{\alpha}-1).$
For $x\neq y$, put $t:=d_{\T}(x,y)$. There is a constant $C_{\alpha}>0$, depending only on $\alpha$, such that
$$|H(x)-H(y)|\le 2\pi t\sum_{m_j\le 1/t}m_j^{1-\alpha}+2\sum_{m_j>1/t}m_j^{-\alpha}\le C_{\alpha}t^{\alpha}.$$
Consequently, $H\in X$ and $\|H\|_{X}\le C_{\alpha}$. 

Set $g\coloneqq f+2\delta H$. The discrete Fourier identity gives 
$$Q_{m_j}(H)=\sum_{l:m_j|m_l}m_l^{-\alpha}\ge m_j^{-\alpha}.$$ 
For every $v\in X$ with $\|v-g\|_X<\delta,$ we have 
\begin{equation}\label{equ:vg}m_j^{\alpha}Q_{m_j}
(v)>-\delta+2\delta-\|v-g\|_X>0,\;\forall j\ge 1.\end{equation}
Given $F\in\Fcal$, take $j=N_F$. Then $\gcd(m_j,N_F)=1$, so $\nu_{m_j}\in\Mcal_F$, and \eqref{equ:vg} gives $v\in G_F$. It follows that the open ball $B(g,\delta)$ is contained in $\bigcap_{F\in\Fcal}(X\cap G_F)$ and hence that $g\in\Ocal_X$. Since $\|g-f\|_X\le2C_{\alpha}\delta$ and $\delta>0$ is arbitrary, $\Ocal_X$ is dense in $X$.

\end{proof}
\section{Proofs of the Main Results}
Throughout this section, let $\Fcal$ be the family of nonempty finite sets $F\subset\N_{\ge1}$ that generate a nonlacunary semigroup. We begin with the following equivalence:
\begin{claim}\label{claim:zeor}
For each $F\in \Fcal$ and $f\in C(\T),$ we have $f\in G_F$ if and only if $h_\mu(T_n)=0$ for every $\mu\in\Mcal_F^{\max}(f)$ and every $n\in\Scal_F.$
\end{claim}
\begin{proof}
Since $T_1x=x\pmod1$ is the identity, it suffices to consider $F\subset\N_{\ge2}$. If Haar measure is a maximizer of $f$, then $h_\lambda (T_n)=\log n>0$ for any $n\in\Scal_F$ with $n\ge2$.  Conversely, let $f\in G_F$  and $\mu\in\Mcal_F^{\max}(f)$. Let $\mu=\int_{\Mcal_F}\nu\,d\Pi(\nu)$ be its ergodic decomposition under the action of $\Scal_F$. Each component has $f$-average at most $\beta_F(f)$, and
$$\int_{\Mcal_F}\left(\beta_F(f)-\int_\T f\;d\nu\right)d\Pi(\nu)
 =\beta_F(f)-\int_\T f\;d\mu=0.
$$
Thus $\Pi$-almost every $\nu$ maximizes $f$ and is distinct from $\lambda$. Rudolph--Johnson rigidity~\cite{Rudolph,Johnson} therefore implies that each such $\Scal_F$-ergodic component has zero entropy for every $T_n$ with $n\in\Scal_F$. For each $n\in\Scal_F$, the components are $T_n$-invariant, so Rokhlin's entropy decomposition formula~\cite[Theorem~9.8]{Rokhlin1967} gives
$$h_\mu(T_n)=\int_{\Mcal_F}h_\nu(T_n)\,d\Pi(\nu)=0.
$$
This proves the claim.
\end{proof}
We now prove Theorem~\ref{thm:main}.
\begin{proof}[Proof of Theorem~\ref{thm:main}] We first prove part (i). Put $X:=C^{k,\alpha}(\T)$. If $f_j\to f$ in $X$ and $f_j\to g$ in $\Bcal$, then both convergences hold uniformly, so $f=g.$ The inclusion $X\hookrightarrow\Bcal$ therefore has a closed graph and hence is continuous by the closed graph theorem. Fix $F\in \Fcal$ and set $\Ocal_F:=\Bcal\cap G_F$, which is open because $\Bcal$ embeds continuously in $C(\T)$ and 
$$|\beta_F(f)-\beta_F(g)|\le \|f-g\|_{\infty},\quad |\int_{\T}(f-g)\;d\lambda|\le \|f-g\|_{\infty}\;.$$
For every nonempty open ball $U\subset \Bcal$, the set $U\cap X$ is nonempty by density and still open in $X$ by continuity. By Lemma~\ref{lm:excluehaar}, $U\cap X$ meets $\Ocal_X\subset G_F$. Thus $\Ocal_F$ is dense in $\Bcal$, and Claim~\ref{claim:zeor} completes the proof of part (i).

For part (ii), fix $F\subset\Fcal$. Note that $\Scal_F$ acts continuously on $\T$ and $\lambda\in\Mcal_F$. The inclusion $\Bcal\hookrightarrow C(\T)$ is continuous by hypothesis, and its image is uniformly dense because $C^{k,\alpha}(\T)\subset\Bcal$ contains the real trigonometric polynomials. Applying~\cite[Theorem~1]{JenkinsonLiLiaoZhang} therefore gives a dense $G_{\delta}$ set $\mathcal U_F\subset \Bcal $ on which the $\Scal_F$-maximizing measure is unique.  Since there are only countably many finite subsets of $\N_{\ge 1},$  Baire's theorem shows that 
 $$\Rcal \coloneqq \bigcap_{F\in \mathcal F}(\Ocal_F\cap \mathcal U_F)$$
 is a dense $G_{\delta}$ set. 
 For every $f\in\Rcal$ and $F\in\mathcal F$, membership in $\mathcal U_F$ gives uniqueness, while membership in $\Ocal_F$ gives the zero-entropy conclusion. 

\end{proof}

\begin{proof}[Proof of Theorem~\ref{thm:holder}]
    Take $X=C^{k,\alpha}(\T)$ and $\Ocal_{k,\alpha}=\Ocal_X$ in Lemma~\ref{lm:excluehaar}. Then Lemma~\ref{lm:excluehaar} and Claim~\ref{claim:zeor} together imply that, for every finitely generated nonlacunary semigroup $\Scal_F$ and every $f\in\Ocal_{k,\alpha}$, all measures in $\Mcal_{F}^{\max}(f)$ have zero entropy.
\end{proof}

The following proposition shows that a common open dense set need not exist in the general Banach space setting of Theorem~\ref{thm:main}.
\begin{proposition}
    \label{prop:uni}
Let $\Fcal$ be the family of nonempty finite sets $F\subset \N_{\ge 1}$ for which $\Scal_F$ is nonlacunary. Then
$$\operatorname{int}_{C(\T)}\left(\bigcap_{F\in \Fcal}G_F\right)=\{f\in C(\T):f(0)>\int_{\T}f\;d\lambda\}=\{Q_1>0\},$$
where $C(\T)$ is equipped with the uniform norm. In particular, the sets $G_F,$ $F\in \Fcal$, contain no common open dense subset of $C(\T)$.
    
\end{proposition}
\begin{proof}
  Set $D\coloneqq \bigcup_{F\in \Fcal}G_{F}^c$.  Each $G_F^c$ is a convex cone, and $G_F^c\subseteq G_{F'}^c$ whenever $F\subseteq F'$.
Since $F_1\cup F_2\in\Fcal$ for all $F_1,F_2\in\Fcal$, the set $D$ is also a convex cone. Moreover, $\delta_0\in\Mcal_F$ for every $F\in\Fcal$, and hence $D\subset\overline D\subset\{Q_1\le0\}$.  

By Jenkinson~\cite[Corollary~1]{Jenkinson2006b}, there exists $h\in C(\T)$ for which $\lambda$ is the unique
$T_2$-maximizing measure. Hence
$$
h\in G_{\{2,3\}}^c\subseteq D,
\qquad Q_1(h)<0.
$$

Suppose that $Q_1(f)\le0$ but $f\notin\overline D$.  By the Hahn--Banach theorem and the Riesz
representation theorem, there is a finite real signed measure $\eta$ such that
$$
\int_\T f\,d\eta>0,
\quad
\int_\T v\,d\eta\le0
\quad\text{for every }v\in D.
$$
The set $D$ contains all constants and all coboundaries $u-u\circ T_n$, hence $\eta(\T)=0
$ and $(T_n)_*\eta=\eta$ for every $n\ge1.$ Write $\hat\eta(k)\coloneqq\int_\T e^{-2\pi ikx}\,d\eta(x)$ for $k\in\Z$. By the invariance of $\eta$ and the dominated convergence theorem, for every $k\ge1$ we have
$$
\hat\eta(k)=\hat\eta(1)
=\lim_{N\to\infty}
\int_\T\frac1N\sum_{n=1}^N  e^{-2\pi inx}\,d\eta(x)
=\eta(\{0\}).
$$
Since $\eta$ is real and $\eta(\T)=0$, the uniqueness of Fourier coefficients yields
$$
\eta=c(\delta_0-\lambda)
\quad\text{for some }c\in\mathbb R.
$$
Applying the separating inequality to $h$ gives $cQ_1(h)=\int_\T h\,d\eta\le0,$
so $c\ge0$. But then $ \int_\T f\,d\eta=c Q_1(f)\le 0,$ a contradiction. Therefore $\overline D=\{Q_1 \le0\}$, and taking complements gives
$$
\operatorname{int}_{C(\T)}
\left(\bigcap_{F\in\Fcal} G_F\right)
=C(\T)\setminus\overline D
=\{Q_1>0\}.
$$

\end{proof}

\begin{remark}\label{rem:finitely-generated}

The finite-generation hypothesis is not optimal, although some restriction on the semigroup is necessary. We do not know an optimal replacement. First, let $\Sigma_1\coloneqq\langle 2\cdot 3^{j}:j\ge0\rangle
$.  Then $\Sigma_1$ is not finitely generated.
In this case, $\Mcal_{\Sigma_1}=\Mcal_{\{2,3\}}$, and hence the conclusion of
Theorem~\ref{thm:main} holds for $\Sigma_1$ as well.

To see that some hypothesis is necessary, consider $\Sigma\coloneqq\N_{\ge1}$. The Fourier argument of Proposition~\ref{prop:uni} gives
$$\Mcal_{\Sigma}=\{a\delta_0+(1-a)\lambda:0\le a\le1\}.$$
Rudolph--Johnson rigidity still applies, but the conclusion of Theorem~\ref{thm:main} fails. Indeed, for $f(x):=-\cos(2\pi x)$, Haar measure is the unique maximizer over $\Mcal_\Sigma$, and $h_\lambda(T_2)=\log2>0$. Consequently, the zero-entropy conclusion fails on the nonempty open set $\{f\in C(\T):\int_{\T}f\,d\lambda>f(0)\}$.
\end{remark}

\section*{Use of artificial intelligence}
During the preparation of this work, GPT-6 was used to explore alternative arguments for Lemma~\ref{lm:holder} and Proposition~\ref{prop:uni} and to assist with English-language editing. The author independently checked all mathematical suggestions and wrote and approved the final manuscript, taking full responsibility for its content.

\section*{Acknowledgments}

The author thanks Professor Rui Gao and Professor Oliver Jenkinson for helpful discussions and valuable suggestions. The author also thanks the anonymous referee of an earlier version of this manuscript for expert suggestions that led to extensions of the results.

\end{document}